\documentclass{amsart}
\usepackage{amsfonts}
\usepackage{amssymb}
\usepackage{amsmath}
\usepackage{amsthm}
\usepackage[dvipdfm]{graphicx,color}
\usepackage{ascmac}
\usepackage{moreverb}
\usepackage{fancybox}
\usepackage{fancyvrb}
\usepackage{comment}
\usepackage[all]{xy}

\newtheorem{thm}{Theorem}[section]
\newtheorem{cor}[thm]{Corollary}
\newtheorem{lem}[thm]{Lemma}

\newtheorem{prob}[thm]{Problem}

\newtheorem*{decconj}{Decomposability Conjecture}

\theoremstyle{definition}
\newtheorem{defin}[thm]{Definition}

\newtheorem{exa}[thm]{Example}

\newtheorem*{claim}{Claim}

\newcommand{\lrangle}[1]{\langle #1 \rangle}
\newcommand{\res}{\upharpoonright}

\newcommand{\dg}[1]{\mathbf{#1}}

\newcommand{\borel}[2]{\mathbf{\Sigma}_{#1,#2}}
\newcommand{\cborel}[3]{\Sigma^{#1}_{#2,#3}}
\newcommand{\dec}[1]{\mathbf{dec}_{#1}}
\newcommand{\cdec}[2]{\mathrm{dec}^{#1}_{#2}}

\title[Decomposing Borel Functions]{Decomposing Borel Functions using the Shore-Slaman Join Theorem}

\author[T. Kihara]{Takayuki Kihara}
\address{School of Information Science\\ Japan Advanced Institute of Science and Technology\\
1-1 Asahidai, Nomi city, Ishikawa, 923-1292 Japan}
\email{kihara@jaist.ac.jp}

\date{}


\begin{document}
\maketitle

\begin{abstract}
Jayne and Rogers proved that every function from an analytic space into a separable metrizable space is decomposable into countably many continuous functions with closed domains if and only if the preimage of each $F_\sigma$ set under it is again $F_\sigma$.
Many researchers conjectured that the Jayne-Rogers theorem can be generalized to all finite levels of Borel functions.
In this paper, by using the Shore-Slaman join theorem on the Turing degrees, we show the following variant of the Jayne-Rogers theorem at finite and transfinite levels of the hierarchy of Borel functions:
For all countable ordinals $\alpha$ and $\beta$ with $\alpha\leq\beta<\alpha\cdot 2$, every function between Polish spaces having small transfinite inductive dimension is decomposable into countably many Baire class $\gamma$ functions with $\mathbf{\Delta}^0_{\beta+1}$ domains such that $\gamma+\alpha\leq\beta$ if and only if, from each $\mathbf{\Sigma}^0_{\alpha+1}$ set, one can continuously find its $\mathbf{\Sigma}^0_{\beta+1}$ preimage.
\end{abstract}


\section{Summary}
\subsection{Introduction}

In the early 20th century, Nikolai Luzin asked whether every Borel function on the real line can be decomposed into countably many continuous functions.
The Luzin problem was negatively answered in the 1930s.
Then, {\em which Borel functions are decomposable into continuous functions?}
In the end of the 19th century, Baire introduced a well-known hierarchy of real functions by iterating pointwise limits of continuous functions.
A famous theorem by Lebesgue and Hausdorff states that every real function is of Baire class $\alpha$ if and only if the preimage of each open set under it is a Borel set of additive class $\alpha$, i.e., a $\mathbf{\Sigma}^0_{\alpha+1}$ set in the well-known Borel hierarchy.
One can introduce finer hierarchy of Borel functions than Baire's one.
For countable ordinals $\alpha,\beta<\omega_1$, a function is called {\em $\borel{\alpha}{\beta}$} if the preimage of each $\mathbf{\Sigma}^0_\alpha$ set under it is $\mathbf{\Sigma}^0_\beta$.
Then, where is the boundary of decomposability in this finer hierarchy of Borel functions?

A remarkable theorem proved by Jayne-Rogers \cite{JR} states that the $\borel{2}{2}$ functions are precisely the $\mathbf{\Delta}^0_2$-piecewise continuous functions, where for a class $\mathbf{\Gamma}$ of Borel sets and a class $\mathcal{F}$ of Borel functions, we say that a function is {\em $\mathbf{\Gamma}$-piecewise $\mathcal{F}$} (denoted by the symbol $\dec{\alpha}\mathcal{F}$ if $\mathbf{\Gamma}$ is a delta class $\mathbf{\Delta}^0_\alpha$) if it is decomposable into countably many $\mathcal{F}$-functions with $\mathbf{\Gamma}$ domains (see also \cite{KMS12} for an alternative proof).
Subsequently, Solecki \cite{Sole98} proved a dichotomy (see also \cite{MRpre,PawSab,Sabok09}) sharpening the Jayne-Rogers theorem by using the Gandy-Harrington topology from effective descriptive set theory.

More recently, a significant breakthrough was made by Semmes \cite{Sem}, who used Wadge-like infinite two-player games and {\em priority arguments} to show that on the zero-dimensional Polish space $\omega^\omega$, the $\borel{3}{3}$ functions are precisely the $\mathbf{\Delta}^0_3$-piecewise continuous functions, and the $\borel{2}{3}$ functions are precisely the $\mathbf{\Delta}^0_3$-piecewise $\mathbf{\Sigma}^0_2$-measurable (i.e., $\borel{1}{2}$) functions.
Countable decomposability at all finite levels of Borel hierarchy has been studied by Pawlikowski-Sabok \cite{PawSab} and Motto Ros \cite{MRpre}.
Naturally, many researchers expected that the Jayne-Rogers theorem and the Semmes theorem could be generalized to all finite levels of the hierarchy of Borel functions (see Andretta \cite{And07}, Semmes \cite{Sem}, Motto Ros \cite[Conjecture 1.6]{MRpre}, and Pawlikowski-Sabok \cite[Conjectures 7.1 and 7.2, and Question 7.3]{PawSab}).

\begin{decconj}
On separable metrizable spaces with analytic domain, the equality $\borel{m+1}{n+1}=\dec{n+1}\borel{1}{n-m+1}$ holds.
In other words, the $\borel{m+1}{n+1}$ functions are precisely the $\mathbf{\Delta}^0_{n+1}$-piecewise $\mathbf{\Sigma}^0_{n-m+1}$-measurable functions at all finite levels $m,n\in\omega$.
\end{decconj}

In this paper, we introduce the notion of the {\em $\borel{\alpha}{\beta}^{\to}$ functions}, which are a special subclass of the $\borel{\alpha}{\beta}$ functions.
Roughly speaking, a function is said to be $\borel{\alpha}{\beta}^\to$ if a continuous function witnesses that it is $\borel{\alpha}{\beta}$, that is, a continuous function maps each code of a $\mathbf{\Sigma}^0_{\alpha}$ set to a code of its $\mathbf{\Sigma}^0_{\beta}$ preimage (For precise definition, see Definition \ref{def:contBorel}).
One can also realize this notion by introducing lightface (i.e., {\em computable}) versions of the $\borel{\alpha}{\beta}$ functions and relativizing them by oracles.

Here, we should emphasize the significance of the concept of decomposability in computability theory and computer science.
As typical examples from computational complexity theory, nonuniform complexity classes are usually defined as classes of problems that are feasibly solvable with advice strings, that is, classes of problems solved by piecewise feasible functions.
For several applications of nonuniform computability on countably-based topological spaces, see \cite{BG,Zie4}.
Moreover, it is important to note that a certain type of computational learning process (such as the identification in the limit) can be captured as $\mathbf{\Delta}^0_2$-piecewise continuity \cite{BMP,dBre13,HiKi,HiKi2}.
Further, as a type of piecewise continuity, the concept of layerwise computability based on Luzin's theorem in measure theory is playing a greater role in the algorithmic randomness theory and effective probability theory \cite{HoyRoj09,Miy}.

Our main theorem states that for all countable ordinals $\alpha$ and $\beta$, every $\borel{\alpha+1}{\beta+1}^\to$ function between Polish spaces having small transfinite inductive dimension is decomposable into a countable list $\{F_n\}_{n\in\omega}$ of functions such that each $F_n$ is $\mathbf{\Sigma}^0_{\gamma+1}$-measurable for some ordinal $\gamma$ with $\gamma+\alpha\leq\beta$.
Furthermore, if $\alpha\leq\beta<\alpha\cdot 2$, a function between Polish spaces having small transfinite inductive dimension is $\borel{\alpha+1}{\beta+1}^\to$ if and only if it is decomposable into such a list $\{F_n\}_{n\in\omega}$ where ${\rm dom}(F_n)$ is $\mathbf{\Delta}^0_{\beta+1}$.
This can be considered as a partial solution to the decomposability conjecture.
To achieve our objective, we employ the Shore-Slaman join theorem on the Turing degrees to show the lightface (i.e., computable) version of our main theorem, and then, we obtain the boldface theorem by relativizing it.

\subsection{Preliminaries}

For the basic concepts of computable analysis, see Weihrauch \cite{Wei}, and for (effective) descriptive set theory, see Kechris \cite{Kec95} and Moschovakis \cite{MosDS}.

The set of all natural numbers is denoted by $\omega$.
The notation $f:\subseteq X\to Y$ denotes that $f$ is a partial function from $X$ into $Y$.
For any reals $X,Y\in\omega^\omega$, the symbol $X\leq_TY$ denotes that $X$ is Turing reducible to $Y$; $X\oplus Y$ denotes the real $Z$ with $Z(2n)=X(n)$ and $Z(2n+1)=Y(n)$.
Given $X\in\omega^\omega$ and $e,n,m\in\omega$, the notation $\Phi_e(X;n)=m$ denotes that the $e$-th Turing machine with input $n$ and oracle $X$ halts and outputs the value $m$.
As usual, we sometimes think of each Turing machine $\Phi_e$ as a partial function from $\omega^\omega$ into $\omega^\omega$, where its domain ${\rm dom}(\Phi_e)$ is the set of all oracles $X$ such that $\Phi_e(X;n)$ is defined for all $n\in\omega$.
Let $X'$ denote the Turing jump of $X$, that is, the halting problem relative to $X$, and let $X^{(\alpha)}$ denote the $\alpha$-th iterated Turing jump of $X$ for every computable ordinal $\alpha$.

Let $\omega^\omega$ denote the Baire space of infinite sequences of natural numbers, that is, the topological product of countably many discrete spaces $\omega$.
Each Borel set is frequently identified with a so-called {\em Borel code} in the fields of (descriptive) set theory.
We only require a coding $B^\alpha:x\mapsto B_x^\alpha$ of $\mathbf{\Sigma}^0_\alpha$ subsets of a given space $\mathcal{X}$ to fulfill the following conditions.
\begin{enumerate}
\item (Total surjectivity) $B^\alpha:\omega^\omega\to\mathbf{\Sigma}^0_\alpha(\mathcal{X})$ is total and surjective.
\item (Measurability) $\{\lrangle{x,y}\in\omega^\omega\times\mathcal{X}:y\in B^\alpha_x\}$ is $\mathbf{\Sigma}^0_\alpha$.
\end{enumerate}

The usual Borel coding restricted to $\mathbf{\Sigma}^0_\alpha$ sets satisfies the above two conditions (see \cite{Bra1,MosDS}).
Hereafter, we fix a Borel coding $B^\alpha$ satisfying (1) and (2), and then, identify each Borel set $B^\alpha_x$ with its code $x\in\omega^\omega$.
For instance, we say that a function $F:\mathbf{\Sigma}^0_\alpha(\mathcal{X})\to\mathbf{\Sigma}^0_\beta(\mathcal{Y})$ is continuous if there is a continuous function $f:\omega^\omega\to\omega^\omega$ such that $F(B^\alpha_x)=B^\beta_{f(x)}$ for every $x\in\omega^\omega$.
Then, the condition (2) can be rephrased as follows.
\begin{enumerate}
\item[(2')]
The membership relation $\in_\alpha:\mathcal{X}\times\mathbf{\Sigma}^0_\alpha(\mathcal{X})\to\mathbb{S}$ is $\mathbf{\Sigma}^{0}_{\alpha}$-measurable,
\end{enumerate}
where $\in_\alpha(x,A)$ is the truth value of $x\in A$, and $\mathbb{S}=\{0,1\}$ is Sierpi\'nski's connected two-point space whose open sets are $\emptyset$, $\{1\}$, and $\{0,1\}$.

Hereafter, by a {\em represented space}, we mean a recursively presented Polish space (see \cite{MosDS}; or, more generally, the reader may take a represented space in this paper to mean an {\em admissibly} represented space in the sense of \cite{Wei}, that is, a $T_0$ quotient of a second-countable space endowed with the notion of computability \cite{Schro}).

\begin{defin}\label{def:contBorel}
Let $X\in 2^\omega$ be a real, let $\alpha,\beta<\omega_1^X$ be ordinals, and let $\mathcal{X}$ and $\mathcal{Y}$ be represented spaces.
A function $F:\mathcal{X}\to\mathcal{Y}$ is $\borel{\alpha}{\beta}^\to$ (respectively, $\cborel{X}{\alpha}{\beta}$) if it is $\borel{\alpha}{\beta}$, and the function $F^{-1}:\mathbf{\Sigma}^0_\alpha(\mathcal{Y})\to\mathbf{\Sigma}^0_\beta(\mathcal{X})$ sending each $\mathbf{\Sigma}^0_\alpha$ set $S\subseteq\mathcal{Y}$ to its preimage $F^{-1}(S)\subseteq\mathcal{X}$ is continuous (respectively, $X$-computable).
\end{defin}

To emphasize its domain and range, we sometimes write $F\in\borel{\alpha}{\beta}^\to(\mathcal{X},\mathcal{Y})$ (respectively, $\cborel{X}{\alpha}{\beta}(\mathcal{X},\mathcal{Y})$).
The inclusion $\borel{\alpha}{\beta}^\to\subseteq\borel{\alpha+\gamma}{\beta+\gamma}^\to$ holds for all ordinals $\alpha,\beta,\gamma<\omega_1$.
A $\borel{1}{\alpha}$ function and a $\cborel{X}{1}{\alpha}$ function are often called a $\mathbf{\Sigma}^0_\alpha$-measurable function and a $\Sigma^{0,X}_\alpha$-computable function, respectively.
The effective hierarchy of Borel functions at finite levels has been studied by Brattka \cite{Bra1}.
Pauly and de Brecht \cite{PBreta} have also studied the {\em Markov-effectivization} of $\borel{2}{2}$ in the sense that $F^{-1}:\Sigma^0_2(\mathcal{Y})\to\Sigma^0_2(\mathcal{X})$ is computable.

\begin{defin}
Let $\mathcal{F}$ be a class of partial functions from a represented space $\mathcal{X}$ into a represented space $\mathcal{Y}$.
\begin{enumerate}
\item 
A function $F:\mathcal{X}\to \mathcal{Y}$ is {\em countably $\mathcal{F}$} (denoted by $\mathbf{dec}\mathcal{F}$) if there is a countable partition $\{Q_i\}_{i\in\omega}$ of $\mathcal{X}$ such that $F\res Q_i\in\mathcal{F}$ for each $i\in\omega$.
Moreover, if each piece $Q_i$ can be chosen as a $\mathbf{\Delta}^0_\alpha$ set, then $F$ is said to be {\em $\mathbf{\Delta}^{0}_{\alpha}$-piecewise $\mathcal{F}$} (denoted by $\dec{\alpha}\mathcal{F}$).
\item
A function $F:\mathcal{X}\to \mathcal{Y}$ is {\em countably $\Sigma^{0,X}_\beta$-computable} (denoted by $\mathrm{dec}\cborel{X}{1}{\beta}$ if there is a countable partition $\{Q_i\}_{i\in\omega}$ of $\mathcal{X}$ such that $F\res Q_i$ is $\Sigma^{0,X}_\beta$-computable uniformly in $i\in\omega$.
Moreover, if $\{Q_i\}_{i\in\omega}$ is uniformly $\Delta^{0,Y}_\alpha$, then $F$ is said to be {\em $\Delta^{0,Y}_\alpha$-piecewise $\Sigma^{0,X}_\beta$-computable} (denoted by $\cdec{Y}{\alpha}\cborel{X}{1}{\beta}$).
\end{enumerate}
\end{defin}

The $\borel{2}{2}^\to$ functions have been studied by Pauly-de Brecht \cite{PBreta}, who showed that two equalities $\cborel{}{2}{2}=\cdec{}{2}\cborel{}{1}{1}$ and $\borel{2}{2}^\to=\dec{2}\borel{1}{1}$ hold.

\begin{exa}
\begin{enumerate}
\item $\cborel{}{1}{\alpha+1}(2^\omega)\not\subseteq\mathbf{dec}\borel{1}{\alpha}(2^\omega)$ holds for each $\alpha<\omega_1^{CK}$.
Indeed, the $\alpha$-th Turing jump $J^{(\alpha)}:2^\omega\to 2^\omega$ is $\Sigma^0_{\alpha+1}$-computable, but is not countably $\mathbf{\Sigma}^0_\alpha$-measurable.
\item Let $\chi_\mathbb{Q}:\mathbb{R}\to 2$ be Dirichlet's nowhere continuous function.
Then, $\chi_\mathbb{Q}\in\cborel{}{3}{3}\cap\cdec{}{3}\cborel{}{1}{1}$, but $\chi_\mathbb{Q}\not\in\borel{1}{2}$.
\end{enumerate}
\end{exa}

\subsection{Main Theorem}

By $\beta\hat{-}\alpha$, we denote the smallest ordinal $\delta$ with $\delta+\alpha>\beta$.
Note that $n\hat{-}m=n-m+1$ for any natural numbers $m\leq n\in\omega$.
To present our main theorem, we use the notation $\borel{1}{(\beta\hat{-}\alpha)}=\bigcup_{\gamma<\beta\hat{-}\alpha}\borel{1}{\gamma+1}$.

\begin{thm}\label{maintheorem0}
Let $\mathcal{X}$ and $\mathcal{Y}$ be Polish spaces having small transfinite inductive dimensions, and let $\alpha\leq\beta<\omega_1$ be countable ordinals.
Then, we have the following inclusions.
\[\dec{\beta+1}\borel{1}{(\beta\hat{-}\alpha)}(\mathcal{X},\mathcal{Y})\subseteq\borel{\alpha+1}{\beta+1}^\to(\mathcal{X},\mathcal{Y})\subseteq\mathbf{dec}\borel{1}{(\beta\hat{-}\alpha)}(\mathcal{X},\mathcal{Y}).\]
\end{thm}

\begin{thm}\label{maintheorem}
Let $\mathcal{X}$ and $\mathcal{Y}$ be Polish spaces having small transfinite inductive dimensions.
For any countable ordinals $\alpha,\beta<\omega_1$ with $\alpha\leq\beta<\alpha\cdot 2$, we have the following equality.
\[\borel{\alpha+1}{\beta+1}^\to(\mathcal{X},\mathcal{Y})=\dec{\beta+1}\borel{1}{(\beta\hat{-}\alpha)}(\mathcal{X},\mathcal{Y}).\]
\end{thm}

Hence, as a corollary we can see that the class $\borel{m+1}{n+1}^\to$ is precisely the class of $\mathbf{\Delta}^0_{n+1}$-piecewise $\mathbf{\Sigma}^0_{n-m+1}$-measurable functions (compare with the decomposability conjecture).
Moreover, if $\alpha\geq\omega$, the assumption of transfinite-dimensionality can be removed from Theorems \ref{maintheorem0} and \ref{maintheorem}.


\section{Proof of Main Theorem}
\subsection{Boldface versus Lightface}

Hereafter, we deal with spaces endowed with the notion of computability which fulfills the fundamental relativization principle that ``continuity is equal to computability relative to an oracle''.
For instance, any represented space in our sense (that is, any recursively presented Polish space, or more generally, any admissibly represented space) satisfies this principle.
It clearly implies the equality $\mathbf{\Sigma}_{1,\alpha}=\borel{1}{\alpha}^\to$ (see also \cite{BrechtY09}), whereas it is still open whether $\mathbf{\Sigma}_{\alpha,\beta}=\borel{\alpha}{\beta}^\to$, in general (see Problem \ref{prob:conttrans}).
The relativization principle also implies the following relativization lemmas for $\borel{\alpha}{\beta}^\to$ and $\dec{\beta}\borel{1}{\alpha}$.

\begin{lem}[Relativization I]\label{lem:relatiI}
Let $\mathcal{X}$ and $\mathcal{Y}$ be represented spaces, and let $\alpha,\beta<\omega_1$ be countable ordinals.
A function $F:\mathcal{X}\to\mathcal{Y}$ is $\borel{\alpha}{\beta}^\to$ if and only if it is $\cborel{X}{\alpha}{\beta}$ for some $X\in 2^\omega$ with $\alpha,\beta<\omega_1^X$.
\qed
\end{lem}

\begin{lem}[Relativization II]\label{prop:piececonlib}
Let $\mathcal{X}$ and $\mathcal{Y}$ be represented spaces, and let $\alpha,\beta<\omega_1$ be a countable ordinal.
A function $F:\mathcal{X}\to\mathcal{Y}$ is $\dec{\beta}\borel{1}{\alpha}$ if and only if it is $\cdec{X}{\beta}\cborel{X}{1}{\alpha}$ for some $X\in 2^\omega$ with $\alpha,\beta<\omega_1^X$.
\qed
\end{lem}


The inclusion $\dec{\beta+1}\borel{1}{(\beta\hat{-}\alpha)}\subseteq\borel{\alpha+1}{\beta+1}^\to$ in Theorem \ref{maintheorem0} can be easily shown by relativizing the following lemma.

\begin{lem}\label{lem:mottoros}
Let $\mathcal{X}$ and $\mathcal{Y}$ be represented spaces.
Fix an oracles $X$, and ordinals $\alpha\leq\beta<\omega_1^X$.
\[\cdec{X}{\beta+1}\cborel{X}{1}{(\beta\hat{-}\alpha)}(\mathcal{X},\mathcal{Y})\subseteq\cborel{X}{\alpha+1}{\beta+1}(\mathcal{X},\mathcal{Y}).\]
\end{lem}

\begin{proof}
Assume that $F:\mathcal{X}\to\mathcal{Y}$ is $\Delta_{\beta+1}^{0,X}$-piecewise $\Sigma^{0,X}_{(\beta\hat{-}\alpha)}$-computable.
Fix an $X$-computable sequence $\{P_e\}_{e\in\omega}$ of $\Delta^{0,X}_{\beta+1}(\mathcal{X})$ sets such that $H_e=F\res P_e$ is $\Sigma^{0,X}_{\gamma(e)+1}(\mathcal{X})$-computable, where $\gamma(e)<\beta\hat{-}\alpha$.
Then, for each $\Sigma^{0,X}_{\alpha+1}(\mathcal{Y})$ set $S\subseteq\mathcal{Y}$, the preimage $F^{-1}(S)$ is the union of $\{H_e^{-1}(S)\cap P_e\}_{e\in\omega}$.
Note that $H_e^{-1}(S)$ is $\Sigma^{0,X}_{\gamma(e)+\alpha+1}$, and the condition $\gamma(e)<\beta\hat{-}\alpha$ implies $\gamma(e)+\alpha\leq\beta$.
Thus, $H_e^{-1}(S)$ is $\Sigma^{0,X}_{\beta+1}$, and its index is computed from any index of $S$ and $e$ by the uniformity.
Thus, $F^{-1}(S)=\bigcup_e(H_e^{-1}(S)\cap P_e)$ is $\Sigma^{0,X}_{\beta+1}$, and we can effectively calculate its index.
Hence, $F$ is a $\cborel{X}{\alpha+1}{\beta+1}$-function.
\end{proof}


\subsection{Shore-Slaman Join Theorem}
The key lemma used to show the inclusion $\borel{\alpha+1}{\beta+1}^\to\subseteq\mathbf{dec}\borel{1}{(\beta\hat{-}\alpha)}$ in Theorem \ref{maintheorem0} is a join theorem concerning the class of $\alpha$-REA operators shown by Shore and Slaman.
We will use the Shore-Slaman join theorem only for the $\alpha$-REA operator $J^{(\alpha)}:x\mapsto x^{(\alpha)}$.

\begin{thm}[Shore-Slaman Join Theorem \cite{ShSl99}]
Let $\alpha$ be a computable ordinal.
The Turing degree structure $(\mathcal{D}_T,\leq,',\oplus)$ satisfies the following formula, for each $k\in\omega$.
\[(\forall\dg{a},\dg{b})(\exists\dg{c}\geq\dg{a})\;[((\forall\beta<\alpha)\;\dg{b}\not\leq\dg{a}^{(\beta)})\;\rightarrow\;(\dg{c}^{(\alpha)}\leq\dg{b}\oplus\dg{a}^{(\alpha)}\leq\dg{b}\oplus\dg{c})].\]
\end{thm}

For $\alpha=1$, it is exactly the Posner-Robinson join theorem \cite{PosnerR81}.
Historically, Jockusch and Shore \cite{JocSho84} were the first to ask whether the Posner-Robinson join theorem can be generalized to all $n$-REA operators for $n\in\omega$.
The main tool for addressing their question was introduced by Kumabe and Slaman, who showed the join theorem for $\alpha=\omega$ (for Kumabe-Slaman forcing, see also Day-Dzhafarov \cite{DayDzh}).
Finally, Shore and Slaman proved the join theorem for all computable ordinals $\alpha$.
It is noteworthy that by combining it with the Slaman-Woodin double jump definability theorem, they showed that the Turing jump is first-order definable in the partial ordering $(\mathcal{D}_T,\leq)$ of the Turing degrees (see Slaman-Woodin \cite{SlWo06}).

We employ the Shore-Slaman join theorem to show our main theorem.
For Theorem \ref{maintheorem0} with $\alpha=\beta$, we only require the Shore-Slaman join theorem for $\alpha=1$, i.e., the Posner-Robinson join theorem.
To show Theorem \ref{maintheorem} on all levels of Borel hierarchy, we need the Shore-Slaman join theorem for all countable ordinals $\alpha<\omega_1$.
By analyzing the proof of Shore-Slaman \cite{ShSl99}, it is not difficult to see that their theorem can be generalized to all countable ordinals $\alpha<\omega_1^X$, for any $X\in2^\omega$.
Here, $\omega_1^X$ is the least countable ordinal that is not computable in $X$.
The relativized Shore-Slaman join theorem implies the following lemma.

\begin{lem}\label{lem:5:meas-to-nonunif}
Let $X\in\omega^\omega$ be a real, and let $\alpha<\omega_1^X$ be a countable ordinal.
Suppose that $(y\oplus Z)^{(\alpha)}\leq_T(x\oplus Z)^{(\beta)}$ for every $Z\geq_TX$.
Then, there exists $\gamma<\beta\hat{-}\alpha$ such that $y\leq_T(x\oplus X)^{(\gamma)}$.
\end{lem}

\begin{proof}
Suppose for the sake of contradiction that $y\not\leq_T(x\oplus X)^{(\gamma)}$ for all $\gamma<\beta\hat{-}\alpha$.
Then, by the Shore-Slaman join theorem relative to $X$, there exists $Z\geq_Tx\oplus X$ such that $Z^{(\beta\hat{-}\alpha)}\leq_Ty\oplus Z$.
Hence, we have
\[(y\oplus Z)^{(\alpha)}\geq_TZ^{(\beta\hat{-}\alpha+\alpha)}>_TZ^{(\beta)}\geq_T(x\oplus X)^{(\beta)}.\]
However, this is a contradiction.
\end{proof}


\subsection{Turing Degree Analysis}

The condition $F^{-1}:\mathbf{\Sigma}^0_{\alpha+1}\to\mathbf{\Sigma}^0_{\beta+1}$ is equivalent to the condition $F^{-1}:\mathbf{\Sigma}^0_\alpha\to\mathbf{\Delta}^0_{\beta+1}$, since $\Pi^{0,X}_{\alpha}\subseteq\Sigma^{0,X}_{\alpha+1}$, for every $\Sigma^{0,X}_\alpha$ set $A\subseteq\mathcal{Y}$, the preimages $F^{-1}(A)$ and $\complement F^{-1}(A)=F^{-1}(\mathcal{Y}\setminus A)$ are $\Sigma^{0,X}_{\beta+1}$.
This proof is clearly effective.
Thus, we can show that, if $F^{-1}:\mathbf{\Sigma}^0_{\alpha+1}\to\mathbf{\Sigma}^0_{\beta+1}$ is $X$-computable, then both $F^{-1}:\mathbf{\Sigma}^0_{\alpha}\to\mathbf{\Sigma}^0_{\beta+1}$ and $\complement F^{-1}:\mathbf{\Sigma}^0_{\alpha}\to\mathbf{\Sigma}^0_{\beta+1}$ are also $X$-computable, where $\complement A$ is the complement of $A$ in the underlying space.

\begin{lem}\label{lem:Post_real}
Let $X\in 2^\omega$ be a real, and let $\alpha,\beta<\omega_1^X$ be ordinals.
Assume that $F:D\to E$ is a $\Sigma_{\alpha+1,\beta+1}^X$ function, where $D$ and $E$ are subsets of $\omega^\omega$.
If $D$ is $\Sigma^{0,X}_{\beta+1}$, $(F(x)\oplus X)^{(\alpha)}\leq_T(x\oplus X)^{(\beta)}$ holds for any $x\in D$.
\end{lem}

\begin{proof}
Note that $S^X_e=\{z\in E:(z\oplus X)^{(\alpha)}(e)=1\}$ is $\Sigma^0_{\alpha}(E)$.
Moreover, the function $S^X:\omega\to\mathbf{\Sigma}^0_{\alpha}(E)$ sending $e$ to $S^X_e$ is $X$-computable.
To determine whether $(F(x)\oplus X)^{(\alpha)}(e)=1$, we note that this condition is equivalent to $F(x)\in S_e^X$, which is also equivalent to $x\in F^{-1}S^X(e)$.
Then, the condition $x\in F^{-1}S^X(e)$ is $\Delta^{0,X}_{\beta+1}$, since $F^{-1}S^X,\complement F^{-1}S^X:\omega\to\mathbf{\Sigma}^0_{\beta+1}$ are $X$-computable, and by the condition (2) of our Borel coding.
Consequently, we obtain the inequality $(F(x)\oplus X)^{(\alpha)}\leq_T(x\oplus X)^{(\beta)}$ for any $x\in D$.
\end{proof}

\begin{lem}\label{lem:5:relative}
Let $\alpha,\beta<\omega_1^X$ be countable ordinals, and let $D$ be a subset of $\omega^\omega$.
Then, a function $F:D\to\omega^\omega$ is of class ${\rm dec}\cborel{X}{1}{\alpha+1}$ if and only if the following condition holds:
\[F(x)\leq_T(x\oplus X)^{(\alpha)},\mbox{ for any }x\in D.\]
\end{lem}

\begin{proof}
($\Rightarrow$)
Fix a countable cover $\{X_i\}_{i\in\omega}$ of $D$ such that $F\res X_i$ is $\Sigma^{0,X}_{\alpha+1}$-computable for each $i\in\omega$.
By the universality of the Turing jump, there is a sequence of indices $\{e(i)\}_{i\in\omega}$ such that for each $i\in\omega$,
\[F(x)=\Phi_{e(i)}((x\oplus X)^{(\alpha)};x),\mbox{ for any }x\in D.\]

($\Leftarrow$)
Conversely, define $Q_e=\{x\in D:\Phi_e((x\oplus X)^{(\alpha)})=F(x)\}$.
For any $x\in D$, if a function satisfies $F(x)\leq_T(x\oplus X)^{(\alpha)}$, there is an algorithm $e\in\omega$ such that $F(x)=\Phi_e((x\oplus X)^{(\alpha)})$.
Therefore, $\bigcup_eQ_e=D$.
Finally, $F_e=\Phi_e((x\oplus X)^{(\alpha)})$ is $\Sigma^{0,X}_{\alpha+1}$-computable for each $e\in\omega$, and $F\res Q_e=F_e\res Q_e$, for each $e\in\omega$, as desired.
\end{proof}

\begin{cor}\label{cor:dec_num}
Let $\alpha,\beta<\omega_1^X$ be countable ordinals, and let $D$ be a $\Sigma^{0,X}_{\beta+1}$ subset of $\omega^\omega$.
Then, we have the following inclusion.
\[\cborel{X}{\alpha+1}{\beta+1}(D,\omega^\omega)\subseteq{\rm dec}\cborel{X}{1}{(\beta\hat{-}\alpha)}(D,\omega^\omega).\]
\end{cor}

\begin{proof}
Fix a $\cborel{X}{\alpha+1}{\beta+1}$ function $F:\omega^\omega\to\omega^\omega$.
Clearly, $F\in\cborel{Z}{\alpha+1}{\beta+1}$ for every $Z\geq_TX$.
By Lemma \ref{lem:Post_real}, the function $F$ must satisfy the inequality
\begin{align*}
(F(x)\oplus Z)^{(\alpha)}\leq_T(x\oplus Z)^{(\beta)}
\end{align*}
for any $Z\geq_T X$ and $x\in D$.
By Lemma \ref{lem:5:meas-to-nonunif}, for any $x\in D$, we have $F(x)\leq_T(x\oplus X)^{(\gamma)}$ for some $\gamma<\beta\hat{-}\alpha$.
Thus, by Lemma \ref{lem:5:relative}, $F$ is of class ${\rm dec}\Sigma^{X}_{1,(\beta\hat{-}\alpha)}$.
\end{proof}


\subsection{Complexity of the Decomposition}

In this section, we assume that $\mathcal{X}=\mathcal{Y}=\omega^\omega$.
Kuratowski's extension theorem states that every partial $\mathbf{\Sigma}^0_{\alpha+1}$-measurable function from an metrizable space into a Polish space can be extended to a $\mathbf{\Sigma}^0_{\alpha+1}$-measurable function with a $\mathbf{\Pi}^0_{\alpha+2}$ domain.
Obviously, this theorem is effectivized as follows.

\begin{claim}
Let $\alpha<\omega_1^X$.
For any partial $\Sigma^{0,X}_{\alpha+1}$-computable function $F:\subseteq\mathcal{X}\to\mathcal{Y}$, there is a $\Pi^{0,X}_{\alpha+2}$ set $D$ with ${\rm dom}(F)\subseteq D\subseteq\mathcal{X}$ and a $\Sigma^{0,X}_{\alpha+1}$-computable extension $G:D\to\mathcal{Y}$ of $F$.
\qed
\end{claim}



\begin{claim}
Every partial $\Sigma^{0,X}_{\gamma+1}$-computable function $F:\subseteq\mathcal{X}\to\mathcal{Y}$ has a total multi-valued $X$-computable extension $\tilde{F}:\mathcal{X}\to\mathbf{\Pi}^0_{\gamma+1}(\mathcal{Y})$ in the sense that $\tilde{F}(x)=\{F(x)\}$ for any $x\in{\rm dom}(F_n)$.
\end{claim}

\begin{proof}
Since $\mathcal{Y}$ is Polish, the diagonal set $\Delta_\mathcal{Y}=\{(x,x)\in\mathcal{Y}^2:x\in\mathcal{Y}\}$ is $\Pi^0_1$.
Note that ${\rm graph}(F)=(F,{\rm id})^{-1}(\Delta_\mathcal{Y})$.
Since $F$ is partially $\Sigma^{0,X}_{\gamma+1}$-computable, there is a $\Pi^{0,X}_{\gamma+1}$ set $G\subseteq\mathcal{X}\times\mathcal{Y}$ such that ${\rm graph}(F)=G\cap({\rm dom}(F)\times\mathcal{Y})$.
Then, the function $\tilde{F}:\mathcal{X}\to\mathbf{\Pi}^0_{\gamma+1}(\mathcal{Y})$ sending $x$ to $G^{[x]}$ is $X$-computable (see Brattka \cite[Proposition 3.2]{Bra1}), where $G^{[x]}=\{y\in\mathcal{Y}:(x,y)\in G\}$.
\end{proof}

We now estimate the complexity of our decomposition.

\begin{lem}\label{lem:main_declemma}
Suppose that $2\leq\alpha\leq \beta< \alpha\cdot 2<\omega_1^X$.
Then, we have the following inclusion:
\[\cborel{X}{\alpha+1}{\beta+1}(\mathcal{X},\mathcal{Y})\cap{\rm dec}\cborel{X}{1}{(\beta\hat{-}\alpha)}(\mathcal{X},\mathcal{Y})\subseteq\cdec{X}{\beta+1}\cborel{X}{1}{(\beta\hat{-}\alpha)}(\mathcal{X},\mathcal{Y}).\]
\end{lem}

\begin{proof}
Assume that $F$ is decomposable into a uniform sequence $\{F_n\}_{n\in\omega}$ of $\Sigma^0_{(\beta\hat{-}\alpha)}$-measurable functions.
It suffices to estimate the complexity of $Q_n=\{x\in{\rm dom}(F_n):F(x)=F_n(x)\}$ in $\mathcal{X}$.

Note that $\alpha\cdot 2=\alpha+\alpha>\beta$ implies that $\beta\hat{-}\alpha\leq\alpha$.
Therefore, $\Pi^{0,X}_{\gamma+1}\subseteq\Sigma^{0,X}_{\alpha+1}$ for any $\gamma<\beta\hat{-}\alpha$.
This implies that the total multi-valued extension $\tilde{F}_n:\mathcal{X}\to\mathbf{\Sigma}^0_{\alpha+1}(\mathcal{Y})$ in the sense of the previous claim is $X$-computable.
Recall that the membership relation $\in_{\beta+1}:\mathcal{X}\times\mathbf{\Sigma}^{0}_{\beta+1}(\mathcal{X})\to\mathbb{S}$ is $\Sigma^{0,X}_{\beta+1}$-computable.
Therefore,
\[K_n=(\in_{\beta+1}\circ({\rm id},F^{-1}\circ\tilde{F_n}))^{-1}(\{1\})=\{(z,x)\in\mathcal{X}^2:F(z)\in\tilde{F_n}(x)\}.\]
is $\Sigma^{0,X}_{\beta+1}$, since $F\in\cborel{X}{\alpha+1}{\beta+1}$ implies that $\in_{\beta+1}\circ({\rm id},F^{-1}\tilde{F_n}):\mathcal{X}^2\to\mathbb{S}$ is $\Sigma^{0,X}_{\beta+1}$-computable, and $\{1\}$ is $\Sigma^0_1$ in $\mathbb{S}$.
Consequently, $Q_n$ can be represented as follows:
\[Q_n={\rm dom}(F_n)\cap({\rm id},{\rm id})^{-1}(K_n\cap\Delta_\mathcal{X})\]

By the first claim, we may assume that ${\rm dom}(F_n)$ is $\Pi^{0,X}_{\gamma+2}$.
Then, $\Pi^{0,X}_{\gamma+2}\subseteq\Pi^{0,X}_{\beta}$ holds since $\alpha\geq 2$ implies that $\gamma<\beta\hat{-}2$.
Consequently, this set is $\Sigma^{0,X}_{\beta+1}$ uniformly in $n\in\omega$.
Let $\{Q_{n,m}\}_{m\in\omega}$ be a uniform sequence of $\Pi^{0,X}_{\beta}$ sets with $Q_n=\bigcup_{m\in\omega}Q_{n,m}$.
Then, $F\res Q_{n}=F_n\res Q_{n,m}$ for each $n,m\in\omega$.
%
\end{proof}

As a consequence, we obtain Theorems \ref{maintheorem0} and \ref{maintheorem} for $\mathcal{X}=\mathcal{Y}=\omega^\omega$, by relativizing Corollary \ref{cor:dec_num} and Lemma \ref{lem:main_declemma} via Lemmas \ref{lem:relatiI} and \ref{prop:piececonlib}.


\subsection{Topological Dimension and Quasi-Polish Spaces}\label{sec:qpolish}

In this section, we discuss the possibility of proving our main theorem for a wider class of topological spaces.
This is an important task because it seems that the original motivations behind pioneering works on first-level Borel isomorphisms (i.e., $\mathbf{\Sigma}_{2,2}$-isomorphisms) by Jayne \cite{Jay74} and Jayne and Rogers \cite{JR} and others were to classify topological spaces.
To show our main theorem for a wider class rather than $\omega^\omega$, we focus on the Borel structure of a given space.

We call a bijection $h:\omega^\omega\to\mathcal{X}$ {\em a Borel isomorphism at the level $3/2$} (for short, a $3/2$-isomorphism) if $h$ is $\mathbf{\Delta}^0_2$-piecewise continuous and $h^{-1}$ is $\mathbf{\Delta}^0_3$-piecewise continuous.
Note that an uncountable (quasi-)Polish space having transfinite small inductive dimension (see Engelking \cite{EngBook}) is $3/2$-isomorphic to $\omega^\omega$ (see also \cite[Theorem 4.21]{MRSchSel13}).

For instance, if a space $\mathcal{X}$ is the Euclidean $n$-space $\mathbb{R}^n$ or the unit $n$-sphere $S^n$ (as a recursively presented Polish space), it is computably $3/2$-isomorphic to Baire space $\omega^\omega$, where a bijection $h:\omega^\omega\to\mathcal{X}$ is called {\em an $X$-computable $3/2$-isomorphism} for some oracle $X\in 2^\omega$ if $h\in\cdec{X}{2}\cborel{X}{1}{1}$ and $h^{-1}\in\cdec{X}{3}\cborel{X}{1}{1}$.
This is because the boundary sphere $\partial B(q;r)$ of each rational open ball is $\Pi^0_1$ uniformly in its center $q$ with ratio $r$, and the $\Pi^0_2$ set $\mathcal{X}\setminus\bigcup_{q,r}\partial B(q;r)$ is computably homeomorphic to a $\Pi^0_1$ subspace of $\omega^\omega$ (see also \cite[Theorems 4.7 and 4.21]{MRSchSel13}).

\begin{cor}\label{maintheorem2}
Let $X\in 2^\omega$ be a real.
Let $\mathcal{X}$ and $\mathcal{Y}$ be represented spaces that are $X$-computably $3/2$-isomorphic to $\omega^\omega$.
For any ordinals $\alpha,\beta<\omega_1^X$ with $\alpha\leq\beta<\alpha\cdot 2$, we have the following equality.
\[\cborel{X}{\alpha+1}{\beta+1}(\mathcal{X},\mathcal{Y})=\cdec{X}{\beta+1}\cborel{X}{1}{(\beta\hat{-}\alpha)}(\mathcal{X},\mathcal{Y}).\]
\end{cor}

\begin{proof}
Assume that $F$ is $\cborel{X}{\alpha+1}{\beta+1}$.
For $\alpha=0$, it is obvious.
If $\alpha=1$, $\alpha\leq \beta<\alpha\cdot 2$ implies $\beta=1$.
Then, it is the computable version of the Jayne-Rogers theorem proved by Pauly-de Brecht \cite{PBreta}.
Thus, we can assume that $\alpha\geq 2$.

Let $h_\mathcal{X}:\omega^\omega\to\mathcal{X}$ and $h_\mathcal{Y}:\omega^\omega\to\mathcal{Y}$ be $X$-computable $3/2$-isomorphisms.
By Lemma \ref{lem:mottoros}, we have $h_\mathcal{X},h_\mathcal{Y}\in\cdec{X}{\alpha+1}\cborel{X}{1}{1}\subseteq\cborel{X}{\alpha+1}{\alpha+1}$ and $h^{-1}_\mathcal{X},h^{-1}_\mathcal{Y}\in\cdec{X}{\beta+1}\cborel{X}{1}{1}\subseteq\cborel{X}{\beta+1}{\beta+1}$.
Assume that $F:\mathcal{X}\to\mathcal{Y}$ is a $\cborel{X}{\alpha+1}{\beta+1}$ function.
It is not hard to see that the function $h_\mathcal{Y}Fh_\mathcal{X}^{-1}:\omega^\omega\to\omega^\omega$ is $\cborel{X}{\alpha+1}{\beta+1}$, since $h_\mathcal{X}^{-1}\in\cborel{X}{\beta+1}{\beta+1}$ and $h_\mathcal{Y}\in\cborel{X}{\alpha+1}{\alpha+1}$.
By Corollary \ref{cor:dec_num}, we can see that $h_\mathcal{Y}Fh_\mathcal{X}^{-1}$ is contained in the class $\cdec{X}{}\cborel{X}{1}{(\beta\hat{-}\alpha)}$.
Then, by Lemma \ref{lem:main_declemma}, we have $h_\mathcal{Y}Fh_\mathcal{X}^{-1}\in\cdec{X}{\beta+1}\cborel{0,X}{1}{(\beta\hat{-}\alpha)}$.
Consequently, $F=h_\mathcal{Y}h_\mathcal{Y}^{-1}Fh_\mathcal{X}h_\mathcal{X}^{-1}\in\cdec{X}{\beta+1}\cborel{X}{1}{(\beta\hat{-}\alpha)}$ holds since $h_\mathcal{Y},h^{-1}_\mathcal{X}\in\cdec{X}{\beta+1}\cborel{X}{1}{1}$.
Conversely, by Lemma \ref{lem:mottoros}, such $F$ is $\cborel{X}{\alpha+1}{\beta+1}$.
\end{proof}

In general, all two uncountable (quasi-)Polish spaces are $\mathbf{\Sigma}^0_3$-measurably isomorphic (\cite[Proposition 4.3]{MRSchSel13}), that is, there is a bijection $h$ between two uncountable (quasi-)Polish spaces such that both $h$ and $h^{-1}$ are $\mathbf{\Sigma}^0_3$-measurable.

\begin{cor}\label{maintheorem2_pp}
Let $X\in 2^\omega$ be a real.
Let $\mathcal{X}$ and $\mathcal{Y}$ be represented spaces that are $\Sigma^{0,X}_n$-computably isomorphic to $\omega^\omega$ for some $n\in\omega$.
For any ordinals $\alpha,\beta<\omega_1^X$ with $\omega\leq\alpha\leq\beta<\alpha\cdot 2$, we have the following equality.
\[\cborel{X}{\alpha+1}{\beta+1}(\mathcal{X},\mathcal{Y})=\cdec{X}{\beta+1}\cborel{X}{1}{(\beta\hat{-}\alpha)}(\mathcal{X},\mathcal{Y}).\]
\end{cor}

\begin{proof}
Let $h_\mathcal{X}:\omega^\omega\to\mathcal{X}$ and $h_\mathcal{Y}:\omega^\omega\to\mathcal{Y}$ be $\Sigma^{0,X}_n$-computable isomorphisms.
By Lemma \ref{lem:mottoros}, we have $h_\mathcal{X},h_\mathcal{X}^{-1},h_\mathcal{Y},h_\mathcal{Y}^{-1}\in\cborel{X}{1}{n}\subseteq\cborel{X}{\omega}{\omega}\subseteq\cborel{X}{\alpha+1}{\alpha+1}$.
Assume that $F:\mathcal{X}\to\mathcal{Y}$ is a $\cborel{X}{\alpha+1}{\beta+1}$ function.
As in the proof of Corollary \ref{maintheorem2}, we can see that $G=h_\mathcal{Y}Fh_\mathcal{X}^{-1}$ is contained in the class $\cdec{X}{\beta+1}\cborel{X}{1}{(\beta\hat{-}\alpha)}$.
If $\alpha\geq\omega$, we now claim that $\beta\hat{-}\alpha$ is a limit ordinal.
If not, there is an ordinal $\gamma$ such that $\beta\hat{-}\alpha=\gamma+1$.
By definition, $\gamma+\alpha\leq\beta$ and note that $1+\alpha=\alpha$
whenever $\alpha\geq\omega$.
Therefore, we have $\beta<\beta\hat{-}\alpha+\alpha=\gamma+1+\alpha=\gamma+\alpha\leq\beta$, a contradiction.

Now, we have a $\Delta^{0,X}_{\beta+1}$ partition $\{P_n\}_{n\in\omega}$ such that each $G\res P_n$ is $\Sigma^{0,X}_{\gamma+1}$-computable for some $\gamma<\beta\hat{-}\alpha$.
Then, it is easy to see that $h_\mathcal{Y}Gh_\mathcal{X}^{-1}\res h_\mathcal{X}(P_n)$ is $\Sigma^{0,X}_{\gamma+2n}$-computable.
Note that $\gamma+2n<\beta\hat{-}\alpha$ holds since $\beta\hat{-}\alpha$ is a limit ordinal.
Consequently, $F=h_\mathcal{Y}Gh_\mathcal{X}^{-1}\in\cdec{X}{\beta+1}\cborel{X}{1}{(\beta\hat{-}\alpha)}$ holds since $\{h_\mathcal{X}(P_n)\}_{n\in\omega}$ is a $\Delta^{0,X}_{\beta+1}$ partition of $\mathcal{X}$.
\end{proof}

As a consequence, we obtain Theorem \ref{maintheorem}, by relativizing Corollaries \ref{maintheorem2} and \ref{maintheorem2_pp} via Lemmas \ref{lem:relatiI} and \ref{prop:piececonlib}.
Indeed, Theorems \ref{maintheorem0} and \ref{maintheorem} also hold for quasi-Polish spaces having transfinite small inductive dimensions (see also de Brecht \cite{Brecht13} for quasi-Polish spaces and the modified Borel hierarchy).

\subsection{Open Questions}

The concept of the $\mathbf{\Sigma}_{\alpha,\beta}$-functions was applied by Jayne \cite{Jay74} to study the Banach space $\mathcal{B}^*_\alpha(X)$ of bounded real-valued Baire functions of class $\alpha$ on a realcompact space $\mathcal{X}$.
Jayne \cite[Theorem 2]{Jay74} showed that for any realcompact spaces $\mathcal{X},\mathcal{Y}$ and  ordinals $\alpha,\beta\geq 1$, $\mathcal{B}_\beta^*(\mathcal{X})$ is linearly isometric to $\mathcal{B}_\alpha^*(\mathcal{Y})$ if and only if there exists a $\borel{\alpha+1}{\beta+1}$-isomorphism of $\mathcal{X}$ onto $\mathcal{Y}$.
Here, a bijection $f:\mathcal{X}\to\mathcal{Y}$ is said to be a $\borel{\alpha+1}{\beta+1}$-isomorphism if $f$ is $\borel{\alpha+1}{\beta+1}$ and its inverse function $f^{-1}$ is $\borel{\beta+1}{\alpha+1}$.
It is natural to ask whether the same result holds for $\borel{\alpha+1}{\beta+1}^\to$-isomorphisms.
The problem is how we refine the result by Jayne \cite[Theorem 1]{Jay74} into the following form.

\begin{prob}
Is every Boolean algebra isomorphism of $\mathbf{\Delta}^0_{\beta+1}(\mathcal{X})$ onto $\mathbf{\Delta}^0_{\alpha+1}(\mathcal{Y})$ induced by a $\borel{\alpha+1}{\beta+1}^\to$-isomorphism of $\mathcal{X}$ onto $\mathcal{Y}$?
\end{prob}

More generally, it is also important to ask whether the classes $\borel{\alpha}{\beta}$ and $\borel{\alpha}{\beta}^\to$ coincide.

\begin{prob}\label{prob:conttrans}
Does the equality $\borel{\alpha}{\beta}(\omega^\omega,\omega^\omega)=\borel{\alpha}{\beta}^\to(\omega^\omega,\omega^\omega)$ hold for all countable ordinals $\alpha,\beta<\omega_1$?
\end{prob}

It should also be asked whether Theorem \ref{maintheorem} can be generalized to all countable ordinals $\alpha,\beta<\omega_1$.
Indeed, Pawlikowski-Sabok \cite[Question 7.3]{PawSab} proposed the problem to find an analogue of the Jayne-Rogers theorem at transfinite levels of Borel functions.
We conclude the paper with a proposal on the precise form of the decomposability problem at transfinite levels of the hierarchy of Borel functions.

\begin{prob}
Let $\mathcal{X}$ and $\mathcal{Y}$ be separable metrizable spaces with $\mathcal{X}$ analytic.
For any countable ordinals $\alpha\leq\beta<\omega_1$, is the following equality true?
\[\borel{\alpha+1}{\beta+1}(\mathcal{X},\mathcal{Y})=\dec{\beta+1}\borel{1}{(\beta\hat{-}\alpha)}(\mathcal{X},\mathcal{Y}).\]
\end{prob}

Recently, Gregoriades and Kihara \cite{GKta} succeeded in removing the continuity assumption from Theorem \ref{maintheorem0}, that is, they showed
\[\dec{\beta+1}\borel{1}{(\beta\hat{-}\alpha)}(\mathcal{X},\mathcal{Y})\subseteq\borel{\alpha+1}{\beta+1}(\mathcal{X},\mathcal{Y})\subseteq\mathbf{dec}\borel{1}{(\beta\hat{-}\alpha)}(\mathcal{X},\mathcal{Y}).\]
in the same cases as the current paper by a slight extension of the current idea.


\subsection*{Acknowledgements}
The author is partially supported by a Grant-in-Aid for JSPS fellows.
The author would like to thank Luca Motto Ros for his discourse on the concept of piecewise definable functions at the Dagstuhl Seminar 11411 entitled ``Computing with Infinite Data: Topological and Logical Foundations''.
The author is also grateful to Matthew de Brecht, Masahiro Kumabe, Andrew Marks, and Arno Pauly for their insightful comments and discussions.
Finally, the author would like to thank the anonymous referees for their valuable comments and suggestions.


\begin{thebibliography}{HD}






\bibitem{And07}
A.~Andretta,
\newblock {\em The {${\sf SLO}$} principle and the {W}adge hierarchy},
\newblock in: Foundations of the Formal Sciences V. Infinite Games,
	S.~Bold et al.~(eds.), Studies in Logic 11, College Publ., London, 2007, 1--38.

\bibitem{Bra1}
V.~Brattka,
\newblock {\em Effective {B}orel measurability and reducibility of functions},
\newblock {MLQ Math. Log. Q.}~51 (2005), 19--44.

\bibitem{BMP}
V.~Brattka, M.~de~Brecht, and A.~Pauly,
\newblock {\em Closed choice and a uniform low basis theorem},
\newblock {Ann. Pure Appl. Logic} 163 (2012), 986--1008.

\bibitem{BG}
V.~Brattka and G.~Gherardi,
\newblock {\em Effective choice and boundedness principles in computable analysis},
\newblock {Bull. Symbolic Logic} 17 (2011), 73--117.

\bibitem{DayDzh}
A.~R.~Day and D.~D.~Dzhafarov,
\newblock {\em Limits to joining with generics and randoms},
\newblock in: Proceedings of the 12th Asian Logic Conference (Wellington, 2011),
	R.~Downey et al.~(eds.), World Scientific, 2013, 76--88.

\bibitem{dBre13}
M.~de~Brecht,
\newblock {\em Levels of discontinuity, limit-computability, and jump operators},
\newblock arXiv:1312.0697.

\bibitem{Brecht13}
M.~de~Brecht,
\newblock {\em Quasi-polish spaces},
\newblock {Ann. Pure Appl. Logic} 164 (2013), 356--381.

\bibitem{BrechtY09}
M.~de~Brecht and A.~Yamamoto,
\newblock {\em {$\Sigma^0_\alpha$}-admissible representations (extended abstract)},
\newblock in: 6th International Conference on Computability and Complexity in Analysis (Ljubljana, 2009),
	A.~Bauer et al.~(eds.), OASIcs 11, Schloss Dagstuhl--Leibniz-Zentrum fuer Informatik, Dagstuhl, 2009, 119--130.

\bibitem{EngBook}
R.~Engelking,
\newblock {\em Dimension theory},
\newblock North-Holland, Amsterdam, 1978.

\bibitem{GKta}
V.~Gregoriades and T.~Kihara,
\newblock {\em On the decomposability conjecture},
\newblock in preparation.

\bibitem{HiKi}
K.~Higuchi and T.~Kihara,
\newblock {\em Inside the {M}uchnik degrees I: Discontinuity, learnability, and constructivism},
\newblock {Ann. Pure Appl. Logic} 165 (2014), 1058--1114.

\bibitem{HiKi2}
K.~Higuchi and T.~Kihara,
\newblock {\em Inside the {M}uchnik degrees II: The degree structures induced by the arithmetical hierarchy of countably continuous functions},
\newblock {Ann. Pure Appl. Logic} 165 (2014), 1201--1241.

\bibitem{HoyRoj09}
M.~Hoyrup and C.~Rojas,
\newblock {\em An application of {M}artin-{L}\"of randomness to effective probability theory},
\newblock in: Mathematical Theory and Computational Practice (Heidelberg, 2009),
	K.~Ambos-Spies et al.~(eds.), Lecture Notes in Comput. Sci.~5635, Springer, Berlin, 2009, 260--269.

\bibitem{Jay74}
J.~E.~Jayne,
\newblock {\em The space of class $\alpha$ {B}aire functions},
\newblock {Bull. Amer. Math. Soc.}~80 (1974), 1151--1156.

\bibitem{JR}
J.~E.~Jayne and C.~A.~Rogers,
\newblock {\em First level {B}orel functions and isomorphism},
\newblock {J. Math. Pure Appl.}~61 (1982), 177--205.

\bibitem{JocSho84}
C.~G.~Jockusch Jr. and R.~A.~Shore,
\newblock {\em Pseudo-jump operators. {II}: Transfinite iterations, hierarchies and minimal covers},
\newblock {J. Symbolic Logic} 49 (1984), 1205--1236.

\bibitem{KMS12}
M.~Ka\v{c}ena, L.~Motto Ros, and B.~Semmes,
\newblock {\em Some Observations on `A new proof of a theorem of {J}ayne and {R}ogers'},
\newblock {Real Anal. Exchange} 38 (2012), 121--132.

\bibitem{Kec95}
A.~S.~Kechris,
\newblock {\em Classical Descriptive Set Theory}, Grad. Texts in Math.~156,
\newblock Springer-Verlag, New York, 1995.

\bibitem{Miy}
K.~Miyabe,
\newblock {\em {$L_1$}-computability, layerwise computability and {S}olovay reducibility},
\newblock {Computability} 2 (2013), 15--29.

\bibitem{MosDS}
Y.~N.~Moschovakis,
\newblock {\em Descriptive Set Theory},
\newblock Math. Surveys Monogr., Amer. Math. Soc., 2009.

\bibitem{MRpre}
L.~Motto~Ros,
\newblock {\em On the structure of finite levels and $\omega$-decomposable {B}orel functions},
\newblock {J. Symbolic Logic} 78 (2013), 1257--1287.

\bibitem{MRSchSel13}
L.~Motto~Ros, P.~Schlicht, and V.~Selivanov,
\newblock {\em Wadge-like reducibilities on arbitrary quasi-Polish spaces},
\newblock to appear in {Math. Structures Comput. Sci}.

\bibitem{PBreta}
A.~Pauly and M.~de~Brecht,
\newblock {\em Non-deterministic computation and the {J}ayne {R}ogers {T}heorem},
\newblock {Electronic Proceedings in Theoretical Computer Science} 143 (2014), 87--96.

\bibitem{PawSab}
J.~Pawlikowski and M.~Sabok,
\newblock {\em Decomposing {B}orel functions and structure at finite levels of the {B}aire hierarchy},
\newblock {Ann. Pure Appl. Logic} 163 (2012), 1748--1764.

\bibitem{PosnerR81}
D.~B.~Posner and R.~W.~Robinson,
\newblock {\em Degrees joining to $\mathbf{0'}$},
\newblock {J. Symbolic Logic} 46 (1981), 714--722.

\bibitem{Sabok09}
M.~Sabok,
\newblock {\em $\sigma$-continuity and related forcings},
\newblock {Arch. Math. Logic} 48 (2009), 449--464.

\bibitem{Schro}
M.~Schr\"oder,
\newblock {\em Admissible Representations for Continuous Computations},
\newblock Ph.D.~Thesis, FenUniversit\"at Hagen, 2003.

\bibitem{Sem}
B.~Semmes,
\newblock {\em A Game for the {B}orel Functions},
\newblock Ph.D.~thesis, Universiteit van Amsterdam, 2009.

\bibitem{ShSl99}
R.~A.~Shore and T.~A.~Slaman,
\newblock {\em Defining the {T}uring jump},
\newblock {Math. Res. Lett.}~6 (1999), 711--722.

\bibitem{SlWo06}
T.~A.~Slaman and W.~Hugh Woodin,
\newblock {\em Definability in degree structures},
\newblock preprint.

\bibitem{Sole98}
S.~Solecki,
\newblock {\em Decomposing {B}orel sets and functions and the structure of {B}aire class $1$ functions},
\newblock {J. Amer. Math. Soc.}~11 (1998), 521--550.

\bibitem{Wei}
K.~Weihrauch,
\newblock {\em Computable Analysis: An Introduction},
\newblock {Texts Theoret. Comput. Sci. EATCS Ser.}~Springer, Berlin, 2000.

\bibitem{Zie4}
M.~Ziegler,
\newblock {\em Real computation with least discrete advice: A complexity theory of nonuniform computability with applications to effective linear algebra},
\newblock {Ann. Pure Appl. Logic} 163 (2012), 1108--1139.

\end{thebibliography}
\end{document}